\documentclass[10pt,a4paper]{amsart}
\usepackage[utf8]{inputenc}
\usepackage{amsmath}
\usepackage{amsfonts}
\usepackage{amssymb}
\usepackage{amsthm}
\usepackage{color}

\newtheorem{theorem}{Theorem}[section]

\newtheorem{proposition}[theorem]{Proposition}

\theoremstyle{definition}

\calclayout

\makeatletter
\@namedef{subjclassname@2020}{%
  \textup{2020} Mathematics Subject Classification}
\makeatother

\begin{document}

\title[Strichartz estimates from decoupling]{On Strichartz estimates from $\ell^2$-decoupling and applications}

\author{Robert Schippa}

\address{Department of Mathematics, Karlsruhe Institute of Technology, 76131, Karlsruhe, Germany} \email{robert.schippa@kit.edu}

%
%
\maketitle

\begin{abstract} Strichartz estimates are derived from $\ell^2$-decoupling for phase functions satisfying a curvature condition. Bilinear refinements without loss in the high frequency are discussed. Estimates are established from uniform curvature generalizing Galilean invariance or from transversality in one dimension. The bilinear refinements are utilized to prove local well-posedness for generalized cubic nonlinear Schr\"odinger equations.
 \end{abstract}

\section{Introduction}
We point out how $\ell^2$-decoupling implies Strichartz estimates for non-degenerate phase functions on tori $\mathbb{T}^n = (\mathbb{R}/2 \pi \mathbb{Z})^n$. These estimates apply to solutions to linear dispersive equations
\begin{equation}
\label{eq:dispersivePDE}
\left\{\begin{array}{cl}
i\partial_t u + \varphi(\nabla/i) u &= 0 ,  \; (t,x) \in \mathbb{R} \times \mathbb{T}^n, \\
u(0) &= u_0, \end{array} \right.
\end{equation}
where $\varphi \in C^2(\mathbb{R}^n,\mathbb{R})$.\\
The eigenvalues of $D^2 \varphi(\xi)$ are denoted by $\{ \gamma_1(\xi),\ldots,\gamma_n(\xi)\}$ and we set
\begin{equation*}
\sigma_\varphi(\xi) = \min( \{ \# neg. \gamma_i(\xi), \# pos. \gamma_i(\xi) \}).
\end{equation*}
The non-degeneracy hypothesis we assume reads as follows:
\begin{align}
\label{eq:EigenvalueBounds}
\exists &\psi: 2^{\mathbb{N}_0} \rightarrow \mathbb{R}^{>0} \text{ such that } \min(|\gamma_i(\xi)|) \sim \max(|\gamma_i(\xi)|) \sim \psi(N), \quad |\xi| \in [N,2N) \tag{$\mathcal{E}^{\sigma_\varphi} (\psi)$} \\
&\text{ and } \sigma_\varphi(\xi) \text{ is independent of } \xi. \notag
\end{align}
By $P_N$ we denote the frequency projector
\begin{equation*}
(P_N f) \widehat (\xi) = 
\begin{cases}
1_{[N,2N)}(|\xi|) \hat{f}(\xi), \quad N \in 2^{\mathbb{N}_0}, \\
1_{[0,1)}(|\xi|) \hat{f}(\xi), \quad N=0 .
\end{cases}
\end{equation*}
The Strichartz estimates we will prove read
\begin{equation}
\label{eq:frequencyLocalizedStrichartzEstimate}
\Vert P_N e^{it \varphi(\nabla/i)} u_0 \Vert_{L^p(I \times \mathbb{T}^n)} \lesssim |I|^{1/p} N^{s(\varphi)} \Vert P_N u_0 \Vert_{L^2}.
\end{equation}

To prove \eqref{eq:frequencyLocalizedStrichartzEstimate}, we will use $\ell^2$-decoupling (cf. \cite{BourgainDemeter2015,BourgainDemeter2017}), more precisely, (variants of) the discrete $L^2$-restriction theorem. This was carried out in \cite{BourgainDemeter2015,BourgainDemeter2017} in the special cases of $\varphi(\xi) = \sum_{i=1}^n \alpha_i \xi_i^2$, $\alpha_i \in \mathbb{R} \backslash 0$. The modest generalization will clarify the role of the asymptotic behaviour of the eigenvalues of $D^2 \varphi$, i.e., the curvature of the characteristic surface of \eqref{eq:dispersivePDE}.

The following proposition is proved:
\begin{proposition}
\label{prop:linearStrichartzEstimates}
Suppose that $\varphi$ satisfies $(\mathcal{E}^k(\psi))$ and let $I \subseteq \mathbb{R}$ be a compact interval. Then, we find the following estimates to hold for any $\varepsilon > 0$:
\begin{equation}
\label{eq:hyperbolicStrichartzEstimates}
\Vert P_N e^{it \varphi(\nabla/i)} u_0 \Vert_{L^p(I \times \mathbb{T}^n)} \lesssim_\varepsilon |I|^{1/p}
\frac{N^{\left( \frac{n}{2} - \frac{n+2}{p} \right)+\varepsilon}}{(\min(\psi(N),1))^{1/p}}, \quad \frac{2(n+2-k)}{n-k} \leq p < \infty.
\end{equation}
\end{proposition}

Recall that certain Strichartz estimates from \cite{Bourgain1993FourierRestrictionNormPhenomenaI,BourgainDemeter2015,BourgainDemeter2017} are known to be sharp up to endpoints. With the above proposition being a generalization, the Strichartz estimates proved above are also sharp in this sense. We shall also consider the example $\varphi(\xi) = |\xi|^a, \quad 1 < a<2$, where the proposition gives an additional loss of derivatives due to decreased curvature compared to the Schr\"odinger case.\\
When we consider the associated nonlinear Schr\"odinger equation, we shall see  why this additional loss does probably not admit relaxation. Moreover, as in \cite{BourgainDemeter2015,BourgainDemeter2017} there are estimates for $2 \leq p \leq \frac{2(n+2-k)}{n-k}$, which follow from interpolation.
In fact, as $p \neq 2$, Proposition \ref{prop:linearStrichartzEstimates} does not yield Strichartz estimates without loss of derivatives.

When we aim to apply these estimates to prove well-posedness of generalized cubic nonlinear Schr\"odinger equations
\begin{equation}
\label{eq:nonlinearSEQ}
\left\{\begin{array}{cl}
i\partial_t u + \varphi(\nabla/i) u &= \pm |u|^2 u ,  \; (t,x) \in \mathbb{R} \times \mathbb{T}^n, \\
u(0) &= u_0 \in H^s(\mathbb{T}^n), \end{array} \right.
\end{equation}
we will use orthogonality considerations to prove bilinear $L^2_{t,x}$-estimates for $High \times Low \rightarrow High$-interaction without loss of derivatives in the high frequency. In \cite[Theorem~3,~p.~193]{BurqGerardTzvetkov2005} was proved the following proposition to derive well-posedness to cubic Schr\"odinger equations on compact manifolds:
\begin{proposition}
\label{prop:orthogonalityWellposedness}
Let $u_0,v_0 \in L^2(\mathbb{T}^n)$, $K,N \in 2^{\mathbb{N}}$. If there exists $s_0 > 0$ such that
\begin{equation}
\label{eq:bilinearStrichartzEstimateI}
\Vert P_N e^{\pm it \varphi(\nabla/i)} u_0 P_K e^{\pm i t \varphi(\nabla/i)} v_0 \Vert_{L^2_{t,x}(I \times \mathbb{T}^n)} \lesssim |I|^{1/2} \min(N,K)^{s_0} \Vert P_N u_0 \Vert_{L^2} \Vert P_K v_0 \Vert_{L^2},
\end{equation}
where $I \subseteq \mathbb{R}$ is a compact time interval with $|I| \gtrsim 1$, then the Cauchy problem \eqref{eq:nonlinearSEQ} is locally well-posed in $H^s$ for $s>s_0$.
\end{proposition}

For $\varphi = \sum_{i=1}^n \alpha_i \xi^2$, \eqref{eq:bilinearStrichartzEstimateI} follows from almost orthogonality and the Galilean transformation (cf. \cite{Bourgain1993FourierRestrictionNormPhenomenaI,Wang2013}). It turns out that it is enough to require \eqref{eq:EigenvalueBounds} to hold for some uniform constant:
\begin{equation}
\label{eq:GlobalEigenvalueBounds}
\exists C_\varphi > 0 : \; \forall \xi \in \mathbb{R}^n : \quad \min(|\gamma_i(\xi)|) \sim \max(|\gamma_i(\xi)|) \sim C_\varphi. \tag{$\mathcal{E}^{\sigma_\varphi}(C_\varphi)$}
\end{equation}

This will be sufficient to generalize the Galilean transformation and prove the following:
\begin{proposition}
\label{prop:generalizedGalileanInvariance}
Suppose that $\varphi \in C^2(\mathbb{R}^n,\mathbb{R})$ satisfies $(\mathcal{E}^k(C_\varphi))$. Then, there is $s(n,k)$ such that we find the estimate
\begin{equation}
\label{eq:BilinearStrichartzEstimateII}
\Vert P_N e^{\pm it \varphi(\nabla/i)} u_0 P_K e^{\pm it \varphi(\nabla/i)} v_0 \Vert_{L^2_{t,x}(I \times \mathbb{T}^n)} \lesssim_{C_\varphi,s} K^{2s} |I|^{1/2} \Vert P_N u_0 \Vert_{L^2} \Vert P_K v_0 \Vert_{L^2}
\end{equation}
to hold for $s>s(n,k)$, where $I \subseteq \mathbb{R}$ denotes a compact time interval, $|I| \gtrsim 1$.
\end{proposition}

This bilinear improvement can also stem from transversality: In \cite{Hani2012,MoyuaVega2008,Schippa2020} short-time bilinear Strichartz estimates were discussed and the following transversality condition played a crucial role in the derivation of the estimates:
\begin{align}
\label{eq:transversalityCondition}
&\text{ There is } \alpha >0 \text{ so that } |\nabla \varphi(\xi_1) \pm \nabla \varphi(\xi_2)| \sim N^{\alpha} \text{ whenever } |\xi_1| \sim K, \; |\xi_2| \sim N, \tag{$\mathcal{T}_\alpha$}
 \\
&K \ll N, K,N \in 2^{\mathbb{N}}. \notag
\end{align}
The corresponding short-time estimate reads
\begin{equation}
\label{eq:shorttimeBilinearStrichartzEstimate}
\Vert P_N e^{\pm it \varphi(\nabla/i)} u_0 P_K e^{\pm it \varphi(\nabla/i)} v_0 \Vert_{L^2_{t,x}([0,N^{-\alpha}],L_x^2(\mathbb{T}))} \lesssim_\varphi N^{-\alpha/2} \Vert P_N u_0 \Vert_{L^2} \Vert P_K v_0 \Vert_{L^2}.
\end{equation}
This is sufficient to derive an $L^2_{t,x}$-estimate for finite times by gluing together the short time intervals:
\begin{proposition}
\label{prop:bilinearFiniteTimeEstimate}
Suppose that $\varphi$ satisfies \eqref{eq:transversalityCondition} and let $K \ll N, K, N \in 2^{\mathbb{N}}$. Then, we find the following estimate to hold:
\begin{equation}
\label{eq:finiteTimeBilinearEstimateTransversality}
\Vert P_N e^{\pm it \varphi(\nabla/i)} u_0 P_K e^{\pm it \varphi(\nabla/i)} v_0 \Vert_{L^2_{t,x}(I \times \mathbb{T})} \lesssim_\varphi |I|^{1/2} \Vert P_N u_0 \Vert_{L^2} \Vert P_K v_0 \Vert_{L^2},
\end{equation}
whenever $I \subseteq \mathbb{R}$ is a compact time interval with $|I| \gtrsim N^{-\alpha}$.
\end{proposition}
\begin{proof}
Let $I = \bigcup_j I_j$, $|I_j| \sim N^{-\alpha}$, where the $I_j$ are disjoint. Then,
\begin{equation*}
\begin{split}
\text{lhs}\eqref{eq:finiteTimeBilinearEstimateTransversality}^2 &\lesssim \sum_{I_j} \Vert P_N e^{\pm it \varphi(\nabla/i)} u_0 P_K e^{\pm it \varphi(\nabla/i)} v_0 \Vert_{L^2_{t,x}(I_j \times \mathbb{T})}^2 \\
&\lesssim (\# I_j) N^{-\alpha} \Vert P_N u_0 \Vert_{L^2}^2 \Vert P_K v_0 \Vert_{L^2}^2
\end{split}
\end{equation*}
and the claim follows from $\# I_j \sim |I| N^\alpha$. 
\end{proof}
Invoking Proposition \ref{prop:orthogonalityWellposedness} together with Propositions \ref{prop:generalizedGalileanInvariance} or \ref{prop:bilinearFiniteTimeEstimate}, the below theorem follows:
\begin{theorem}
\label{thm:localWellposednessGeneralizedNLS}
Suppose that $\varphi \in C^2(\mathbb{R}^n,\mathbb{R})$ satisfies $(\mathcal{E}^k(C_{\varphi}))$. Then, there is $s_0(n,k)$ such that \eqref{eq:nonlinearSEQ} is locally well-posed for $s>s_0(n,k)$.\\
Let $n=1$ and suppose that $\varphi$ satisfies \eqref{eq:transversalityCondition}. Then, there is $s_0=s_0(\varphi)$ such that \eqref{eq:nonlinearSEQ} is locally well-posed for $s>s_0(\varphi)$.
\end{theorem}

Finally, we give examples: In one dimension we treat the fractional Schr\"odinger equation
\begin{equation}
\label{eq:fractionalSEQ}
\left\{\begin{array}{cl}
i\partial_t u + D^a u &= \pm |u|^2 u ,  \; (t,x) \in \mathbb{R} \times \mathbb{T}, \\
u(0) &= u_0 \in H^s(\mathbb{T}), \end{array} \right.
\end{equation}
where $D= (-\Delta)^{1/2}$.\\
Theorem \ref{thm:localWellposednessGeneralizedNLS} yields uniform local well-posedness for $s> \frac{2-a}{4}$, $1 <a < 2$, which is presumably sharp up to endpoints as discussed in \cite{ChoHwangKwonLee2015}, where the endpoint $s= \frac{2-a}{4}$ was covered by resonance considerations.

For $0 < a < 1$, Dinh \cite{Dinh2017} showed local well-posedness for $s> \frac{2-a}{4}$ via Strichartz estimates for fractional Schr\"odinger equations on compact manifolds. In \cite{Dinh2017} short-time arguments were used to derive Strichartz estimates on arbitray compact manifolds. These estimates we can improve on tori for $1<a<2$ because we do not have to sum up frequency dependent time intervals. Since the derived estimates essentially resemble the estimates on Euclidean space (cf. \cite{HongSire2015}), we conjecture the estimates to be sharp up to endpoints.

We also discuss hyperbolic Schr\"odinger equations. The well-posedness result from \cite{Wang2013,GodetTzvetkov2012} is recovered for the hyperbolic nonlinear Schr\"odinger equation in two dimensions, which is known to be sharp up to endpoints. Generalizing the example, which probes sharpness in higher dimensions, indicates that there is only a significant difference between hyperbolic and elliptic Schr\"odinger equations in low dimensions.

This note is structured as follows: In Section \ref{section:LinearStrichartzEstimates} we prove linear Strichartz estimates utilizing $\ell^2$-decoupling, in Section \ref{section:BilinearStrichartzEstimates} we discuss bilinear Strichartz estimates without loss in the high frequency, and in Section \ref{section:WellposednessGeneralizedSEQ} the implied well-posedness results for generalized cubic nonlinear Schr\"odinger equations are discussed.
\section{Linear Strichartz Estimates}
\label{section:LinearStrichartzEstimates}
We prove Proposition \ref{prop:linearStrichartzEstimates} generalizing the proofs from \cite{BourgainDemeter2015,BourgainDemeter2017}:
\begin{proof}[Proposition \ref{prop:linearStrichartzEstimates}]
Without loss of generality let $I=[0,T]$. First, let $p \geq \frac{2(n+2-k)}{n-k}$ and compute
\begin{equation*}
\begin{split}
\text{lhs}\eqref{eq:hyperbolicStrichartzEstimates}^p &= \int_{\substack{0 \leq x_1, \ldots, x_n \leq 2 \pi,\\ 0 \leq t \leq T}} \left| \sum_{|\xi| \sim N} e^{i(x.\xi+t \varphi(\xi))} \hat{u}_0(\xi) \right|^p dx dt \\
&\sim \frac{N^{-(n+2)}}{\psi(N)} \int_{\substack{0 \leq x_1, \ldots, x_n \leq N, \\ 0 \leq t \leq TN^2 \psi(N)}} \left| \sum_{\substack{|\xi| \sim 1, \xi \in \mathbb{Z}^n/N}} e^{i(x.\xi + \frac{t}{N^2 \psi(N)}\varphi(N\xi))} \hat{u}_0(N\xi) \right|^p dx dt.
\end{split}
\end{equation*}

We distinguish between $\psi(N) \ll 1$ and $\psi(N) \gtrsim 1$. In the latter case, we use periodicity in space to find for the above display
\begin{equation*}
\sim \frac{N^{-(n+2)}}{(TN \psi(N))^n \psi(N)} \int_{\substack{0 \leq x_1, \ldots, x_n \leq T N^2 \psi(N),\\ 0 \leq t \leq TN^2 \psi(N)}} \left| \sum_{\substack{|\xi| \sim 1,\\ \xi \in \mathbb{Z}^n/N}} \hat{u}_0(N \xi) e^{i(x.\xi + \frac{t}{N^2 \psi(N)} \varphi(N\xi))} \right|^p dx dt.
\end{equation*}

This expression is amenable to the discrete $L^2$-restriction theorem \cite[Theorem~2.1,~p.~354]{BourgainDemeter2015} or the variant for hyperboloids because $TN^2 \psi(N) \gtrsim N^2$, and the frequency points are separated of size $\frac{1}{N}$ and the eigenvalues of $\frac{\varphi(N \cdot)}{N^2 \psi(N)}$ are approximately one.

Hence, we have the following estimate uniform in $\varphi$ (the dependence is encoded in $\psi(N)$, which drops out in the ultimate estimate):
\begin{equation*}
\begin{split}
\text{lhs}\eqref{eq:hyperbolicStrichartzEstimates}^p &\lesssim_\varepsilon \frac{N^{-(n+2)}}{(TN \psi(N))^n \psi(N)} (TN^2 \psi(N))^{n+1} N^{\left( \frac{n}{2} - \frac{n+2}{p} \right)p + \varepsilon} \Vert P_N u_0 \Vert^p_{2} \\
&\lesssim T N^{\left( \frac{n}{2} - \frac{n+2}{p} \right)p + \varepsilon} \Vert P_N u_0 \Vert_2^p.
\end{split}
\end{equation*}

Next, suppose that $\psi(N) \ll 1$. In this case we can not avoid loss of derivatives in general. Following along the above lines, we find for $p \geq \frac{2(n+2-k)}{n-k}$
\begin{equation*}
\begin{split}
\text{lhs}\eqref{eq:hyperbolicStrichartzEstimates}^p &\sim \frac{N^{-(n+2)}}{\psi(N)} \int_{\substack{0 \leq x_1, \ldots, x_n \leq N,\\ 0 \leq t \leq TN^2 \psi(N)}} \left| \sum_{\substack{|\xi| \sim 1,\\ \xi \in \mathbb{Z}^n/N}} e^{i(x.\xi+ t \frac{\varphi(N \xi)}{N^2 \psi(N)})} \hat{u}_0(N \xi) \right|^p dx dt \\
&\lesssim \frac{N^{-(n+2)}}{(NT)^n \psi(N)} \int_{\substack{0 \leq x_1, \ldots, x_n \leq TN^2,\\ 0 \leq t \leq TN^2}} \left| \sum e^{i(x.\xi + \frac{t \varphi(N\xi)}{N^2 \psi(N)})} \hat{u}_0(N\xi) \right|^p dx dt \\
&\lesssim_\varepsilon \frac{T}{\psi(N)} N^{\left( \frac{n}{2} - \frac{n+2}{p} \right)p + \varepsilon} \Vert P_N u_0 \Vert_2^p,
\end{split}
\end{equation*}
which yields the claim.
\end{proof}

As an example consider Strichartz estimates for the free fractional Schr\"odinger equation
\begin{equation}
\label{eq:freeFractionalSEQ}
\left\{\begin{array}{cl}
i\partial_t u + D^a u &= 0 ,  \; (t,x) \in \mathbb{R} \times \mathbb{T}^n, \\
u(0) &= u_0. \end{array} \right.
\end{equation}

The phase function $\varphi(\xi) = |\xi|^a, \; 1 < a < 2$ is elliptic and the lack of higher differentiability in the origin is not an issue because low frequencies can always be treated with Bernstein's inequality. $\varphi$ satisfies $(\mathcal{E}^0(\psi))$ with $\psi(N)=N^{a-2}$, hence we find by virtue of Proposition \ref{prop:linearStrichartzEstimates}
\begin{equation}
\label{eq:StrichartzEstimatesFractionalSEQ}
\Vert e^{it D^a} u_0 \Vert_{L^4_{t,x}(I \times \mathbb{T}^n)} \lesssim_{n,a,s} |I|^{1/4} \Vert u_0 \Vert_{H^s}, \quad s> s_0 =
\begin{cases}
\frac{2-a}{8}, \quad n=1,\\
\frac{2-a}{4} + \left( \frac{n}{2} - \frac{n+2}{4} \right), \quad \text{else}.
\end{cases}
\end{equation}

To find the $L^4_{t,x}$-estimate in one dimension, we interpolate the $L^6_{t,x}$-endpoint estimate with the trivial $L^2_{t,x}$-estimate.

In case $n=1, \; 1<a<2$, this recovers the Strichartz estimates from \cite{DemirbasErdoganTzirakis2016}, and for $0<a<1$, this estimate was proved in \cite{Dinh2017}. For $n \geq 2, \; 1< a <2$ this estimate seems to be new. Comparing with the results in Euclidean space (cf. \cite{HongSire2015}), the estimates are presumably sharp up to endpoints.
\section{Bilinear Strichartz estimates and transversality}
\label{section:BilinearStrichartzEstimates}
The argument from Section \ref{section:LinearStrichartzEstimates} admits bilinearization provided that the dispersion relation satisfies \eqref{eq:GlobalEigenvalueBounds}. This generalizes Galilean invariance, which was previously used to infer a bilinear estimate with no loss in the high frequency (cf. \cite{Bourgain1993FourierRestrictionNormPhenomenaI,Wang2013}).
\begin{proof}[Proposition \ref{prop:generalizedGalileanInvariance}]
Let $P_N = \sum_{K_1} R_{K_1}$, where $R_K$ projects to cubes of sidelength $K$. Then, by means of almost orthogonality
\begin{equation*}
\begin{split}
&\quad \Vert P_N e^{it \varphi(\nabla/i)} u_0 P_K e^{it \varphi(\nabla/i)} v_0 \Vert_{L^2_{t,x}(I \times \mathbb{T}^n)}^2 \\
&\lesssim \sum_{K_1} \Vert R_{K_1} e^{it \varphi(\nabla/i)} u_0 P_K e^{it \varphi(\nabla/i)} v_0 \Vert^2_{L^2_{t,x}(I \times \mathbb{T}^n)}.
\end{split}
\end{equation*}

After applying H\"older's inequality, we are left with estimating two $L^4_{t,x}$-norms. Clearly, by Proposition \ref{prop:linearStrichartzEstimates},
\begin{equation*}
\Vert P_K e^{it \varphi(\nabla/i)} v_0 \Vert_{L^4_{t,x}(I \times \mathbb{T}^n)} \lesssim_{\varphi,s} K^{s} \Vert P_K v_0 \Vert_{L^2}
\end{equation*}
provided that $s > s(n,\sigma_\varphi)$.

To treat the other term, let $\xi_0$ denote the center of the cube $Q_{K_1}$ onto which $R_{K_1}$ is projecting in frequency space. Following along the above lines, we write
\begin{equation*}
\begin{split}
&\quad \Vert R_{K_1} e^{it \varphi(\nabla/i)} u_0 \Vert^4_{L^4_{t,x}(I \times \mathbb{T}^n)} \\
 &= \int_{\substack{0 \leq x_1, \ldots, x_n \leq 2 \pi,\\ 0 \leq t \leq T}} \left| \sum_{\xi \in Q_{K_1}} e^{i(x.\xi+t \varphi(\xi))} \hat{u}_0(\xi) \right|^4 dx dt \\
&= \int_{\substack{0 \leq x_1, \ldots, x_n \leq 2 \pi,\\ 0 \leq t \leq T}} \left| \sum_{|\xi^\prime| \leq K} \hat{u}_0(\xi+\xi^\prime) e^{i(x.(\xi_0+\xi^\prime) + t \varphi(\xi_0+\xi^\prime))} \right|^4 dx dt \\
&= \int_{\substack{0 \leq x_1, \ldots, x_n \leq 2 \pi,\\ 0 \leq t \leq T}} \left| \sum_{|\xi^\prime| \leq K} e^{i((x+t \nabla \varphi(\xi_0)).\xi^\prime+t \psi_{\xi_0}(\xi^\prime))} \hat{w}_0(\xi^\prime) \right|^4 dx dt \\
&= \Vert P_{\leq K_1} e^{it \psi_{\xi_0}(\nabla/i)} w_0(x+t \nabla \varphi(\xi_0)) \Vert^4_{L^4(I \times \mathbb{T}^n)},
\end{split}
\end{equation*}
where $\psi_{\xi_0}(\xi^\prime) = \varphi(\xi_0+\xi^\prime) - \varphi(\xi_0) - \nabla \varphi(\xi_0).\xi^\prime$.

After breaking $\Vert P_{\leq K} e^{it \psi_{\xi_0}(\nabla/i)} w_0 \Vert_{L^4_{t,x}(I \times \mathbb{T}^n)} \leq \sum_{1 \leq L \leq K} \Vert P_L e^{it \psi_{\xi_0}(\nabla/i)} w_0 \Vert_{L^4}$, the single expressions are amenable to Proposition \ref{prop:linearStrichartzEstimates}. Indeed, the size of the moduli of the eigenvalues of $D^2 \psi_{\xi_0}$ are approximately independent of the frequencies. Hence,
\begin{equation*}
	\Vert P_L e^{it \psi_{\xi_0}(\nabla/i)} w_0 \Vert_{L^4_{t,x}(I \times \mathbb{T}^n)} \lesssim_{\varepsilon,C_\varphi} L^{s(n,k)+\varepsilon} \Vert P_L w_0 \Vert_{L^2}.
\end{equation*}
Carrying out the sum and the relation of $u_0$ and $w_0$, we find 
\begin{equation*}
\Vert P_{\leq K} e^{it \psi_{\xi_0}(\nabla/i)} w_0 \Vert_{L^4(I \times \mathbb{T}^n)} \lesssim_{\varepsilon,\varphi} K^{s(n,k)+\varepsilon} \Vert R_{K_1} u_0 \Vert_{L^2}
\end{equation*}
By almost orthogonality,
\begin{equation*}
\left( \sum_{K_1} \Vert R_{K_1} u_0 \Vert^2_{L^2} \right)^{1/2} \lesssim \Vert P_N u_0 \Vert_{L^2}^2,
\end{equation*}
and the proof is complete.
\end{proof}
In one dimension (and for certain phase functions also in higher dimensions, see \cite{Schippa2020}) transversality \eqref{eq:transversalityCondition} of the phase function allows us to derive Proposition \ref{prop:bilinearFiniteTimeEstimate}, which improves the above estimate.
\section{Local well-posedness of generalized cubic Schr\"odinger equation}
\label{section:WellposednessGeneralizedSEQ}
Deploying Proposition \ref{prop:orthogonalityWellposedness} by use of the estimates from Section \ref{section:LinearStrichartzEstimates} and \ref{section:BilinearStrichartzEstimates}, we can conclude the proof of Theorem \ref{thm:localWellposednessGeneralizedNLS}:
\begin{proof}
First, suppose that $\varphi$ satisfies \eqref{eq:GlobalEigenvalueBounds}. In case $K \ll N$, Proposition \ref{prop:generalizedGalileanInvariance} yields the estimate
\begin{equation}
\label{eq:BilinearStrichartzEstimateIII}
\begin{split}
&\quad \Vert P_N e^{\pm it \varphi(\nabla/i)} u_0 P_K e^{\pm it \varphi(\nabla/i)} v_0 \Vert_{L^2_{t,x}(I \times \mathbb{T}^n)} \\
&\lesssim_{\varepsilon,\varphi} |I|^{1/2} K^{2s(n,\sigma_\varphi)+\varepsilon} \Vert P_N u_0 \Vert_{L^2} \Vert P_K v_0 \Vert_{L^2}.
\end{split}
\end{equation}

For $K \sim N$, the claim follows after applying H\"older's inequality and Proposition \ref{prop:linearStrichartzEstimates}. From Proposition \ref{prop:orthogonalityWellposedness} we find \eqref{eq:nonlinearSEQ} to be locally well-posed provided that $s>2s_0(n,\sigma_\varphi)$. In case $\varphi$ satisfies $(\mathcal{E}^0(\psi(N)))$ and \eqref{eq:transversalityCondition}, we have the improved bilinear bound
\begin{equation*}
\Vert P_N e^{\pm it \varphi(\nabla/i)} u_0 P_K e^{\pm it \varphi(\nabla/i)} v_0 \Vert_{L^2_{t,x}(I \times \mathbb{T})} \lesssim_{\varphi} |I|^{1/2} \Vert P_N u_0 \Vert_{L^2} \Vert P_K v_0 \Vert_{L^2}
\end{equation*}
due to Proposition \ref{prop:bilinearFiniteTimeEstimate}. Hence, only loss stems from $High \times High \rightarrow High$-interaction, where $K \sim N$: By means of Proposition \ref{prop:linearStrichartzEstimates} and H\"older's inequality, we derive
\begin{equation*}
\Vert P_N e^{\pm it \varphi(\nabla/i)} u_0 P_K e^{\pm it \varphi(\nabla/i)} v_0 \Vert_{L^2_{t,x}(I \times \mathbb{T})} \lesssim_\varphi K^{2s} |I|^{1/2} \Vert P_N u_0 \Vert_{L^2} \Vert P_K v_0 \Vert_{L^2}
\end{equation*}
and by Proposition \ref{prop:orthogonalityWellposedness} we find \eqref{eq:nonlinearSEQ} to be locally well-posed for $s>2s_0(\varphi)$.
\end{proof}
We turn to examples: As discussed in Section \ref{section:LinearStrichartzEstimates}, the phase functions $\varphi(\xi) = |\xi|^a \quad (0<a<2, \; a \neq 1)$ do not satisfy $(\mathcal{E}^0(C_\varphi))$, but $(\mathcal{T}_{a-1})$ for $1<a<2$. For $0<a<1$, $K \ll N$, $|I| \gtrsim 1$, we have the following bilinear Strichartz estimate
\begin{equation*}
\Vert P_N e^{\pm it D^a} u_0 P_K e^{\pm it D^a} v_0 \Vert_{L^2_{t,x}(I \times \mathbb{T})} \lesssim |I|^{1/2} K^{\frac{1-a}{2}} \Vert P_N u_0 \Vert_{L^2} \Vert P_K v_0 \Vert_{L^2},
\end{equation*}
which can be proved like in \cite{MoyuaVega2008,Schippa2020}.

Consequently, by (the proof of) Theorem \ref{thm:localWellposednessGeneralizedNLS} we find \eqref{eq:fractionalSEQ} to be locally well-posed for $s> \frac{2-a}{4}$. As discussed in \cite{ChoHwangKwonLee2015}, this is likely to be the threshold of uniform local well-posedness, which indicates that the linear Strichartz estimates from Section \ref{section:LinearStrichartzEstimates} are in this case sharp up to endpoints.

In Euclidean space fractional Schr\"odinger equations were considered in \cite{HongSire2015}. Key ingredient to well-posedness are linear and bilinear Strichartz estimates, which hold globally in time due to dispersive effects. On the circle we can reach the same regularity like in \cite{HongSire2015} up to the endpoint. Although the linear Strichartz estimates might well be sharp in higher dimensions, satisfactory bilinear $L^2_{t,x}$-Strichartz estimates appear to be beyond the methods of this paper so that we can not prove non-trivial well-posedness results in higher dimensions. Progress presumably requires an additional angular decomposition (cp. \cite{CollianderKeelStaffilaniTakaokaTao2008}) to improve control of the resonance function.

For hyperbolic phase functions, Theorem \ref{thm:localWellposednessGeneralizedNLS} recovers the results from \cite{Wang2013,GodetTzvetkov2012}, where essentially sharp local well-posedness of
\begin{equation}
\label{eq:hyperbolicSEQI}
\left\{\begin{array}{cl}
i\partial_t u + ( \partial_{x_1}^2 - \partial_{x_2}^2) u &= \pm |u|^2 u ,  \; (t,x) \in \mathbb{R} \times \mathbb{T}^2, \\
u(0) &= u_0 \end{array} \right.
\end{equation}
was proved for $s>1/2$. Notably, due to subcriticality of the $L^4_{t,x}$-Strichartz estimate already for the hyperbolic equations
\begin{equation}
\label{eq:hyperbolicSEQII}
\left\{\begin{array}{cl}
i\partial_t u + (\partial_{x_1}^2 - \partial_{x_2}^2 + \partial_{x_3}^2) u &= \pm |u|^2 u ,  \; (t,x) \in \mathbb{R} \times \mathbb{T}^3, \\
u(0) &= u_0, \end{array} \right.
\end{equation}
and
\begin{equation}
\label{eq:hyperbolicSEQIII}
\left\{\begin{array}{cl}
i\partial_t u + (\partial_{x_1}^2 - \partial_{x_2}^2 + \partial_{x_3}^2 - \partial_{x_4}^2) u &= \pm |u|^2 u ,  \; (t,x) \in \mathbb{R} \times \mathbb{T}^4, \\
u(0) &= u_0, \end{array} \right.
\end{equation}
the (essentially sharp) Strichartz estimates yield the same well-posedness results as for the elliptic counterparts.

Firstly, recall the counterexample from \cite{Wang2013}, which showed $C^3$-ill-posedness of \eqref{eq:hyperbolicSEQI} for $s<1/2$. As initial data consider
\begin{equation*}
\phi_N(x) = N^{-1/2} \sum_{|k|\leq N} e^{ik x_1} e^{-ik x_2},
\end{equation*}
which satisfies $\Vert \phi_N \Vert_{H^s} \sim N^s$ and $S[\phi_N](t):= e^{it(\partial_{x_1}^2 - \partial_{x_2}^2)} \phi_N = \phi_N$. This implies
\begin{equation*}
\left\Vert \int_0^T |S[\phi_N](s)|^2 S[\phi_N](s) ds \right\Vert_{H^s} = T \Vert |\phi_N|^2 \phi_N \Vert_{H^s} \gtrsim T N^{1+s}
\end{equation*}
For details on this estimate see \cite{Wang2013}.

The veracity of the estimate
\begin{equation*}
\left\Vert \int_0^T |S[\phi_N](s)|^2 S[\phi_N](s) ds \right\Vert_{H^s} \lesssim \Vert \phi_N \Vert_{H^s}^3 \quad (T \lesssim 1)
\end{equation*}
requires $s \geq 1/2$.

The same counterexample shows that $s \geq 1/2$ is required for $C^3$-well-posedness of \eqref{eq:hyperbolicSEQII}. This regularity is reached up to the endpoint by Theorem \ref{thm:localWellposednessGeneralizedNLS}.\\
When considering \eqref{eq:hyperbolicSEQIII}, we modify the above example to
\begin{equation*}
\phi_N(x) = N^{-1} \sum_{|k_1|,|k_2| \leq N} e^{i k_1 x_1} e^{-i k_1 x_2} e^{i k_2 x_3} e^{-i k_2 x_4},
\end{equation*}
which again satisfies $\Vert \phi_N \Vert_{H^s} \sim N^s$.\\
Carrying out the estimate for the first Picard iterate with the necessary modifications yields
\begin{equation*}
\left\Vert \int_0^T |S[\phi_N](s)|^2 S[\phi_N](s) ds \right\Vert_{H^s} = T \Vert |\phi_N|^2 \phi_N \Vert_{H^s} \gtrsim T N^{2+s},
\end{equation*}
which implies $C^3$-ill-posedness unless $s \geq 1$. This regularity is again obtained up to the endpoint by Theorem \ref{thm:localWellposednessGeneralizedNLS}.

Apparently, for other hyperbolic Schr\"odinger equations the $L^4_{t,x}$-Strichartz estimate also coincides with the elliptic $L^4_{t,x}$-estimate and modifications of the above counterexample yield lower thresholds than in the elliptic case. This indicates that the difference between elliptic and hyperbolic Schr\"odinger equations is only significant for lower dimensions.
\section*{Acknowledgements}
Financial support by the German Science Foundation (IRTG 2235) is gratefully acknowledged.
\bibliographystyle{plain}

\end{document}